\documentclass[a4paper, 11pt]{article}\usepackage{amsmath, amsthm, amssymb, array, mathrsfs, enumerate, enumitem, xcolor}
\IfFileExists{ulem.sty}{\usepackage[normalem]{ulem}}{} 

\usepackage{authblk}
\usepackage{indentfirst}
\usepackage{cite}
\usepackage{hyperref}
\usepackage{cleveref}

\DeclareMathOperator*{\esssup}{ess\,sup}

\newtheorem{theorem}{Theorem}[section]
\newtheorem{lemma}{Lemma}[section]

\newtheorem{assumption}{Assumption}[section]
\newtheorem{remark}{Remark}[section]

\crefname{theorem}{Theorem}{Theorems}
\crefname{lemma}{Lemma}{Lemmas}
\crefname{proposition}{Proposition}{Propositions}
\crefname{corollary}{Corollary}{Corollaries}
\crefname{definition}{Definition}{Definitions}
\crefname{example}{Example}{Examples}
\crefname{remark}{Remark}{Remarks}
\crefname{assumption}{Assumption}{Assumptions}

\def\1{\mathbf{1}}

\def\dif{\mathrm{d}}
\def\e{\mathrm{e}}
\def\E{\mathbb{E}}
\def\Ex{\mathbf{E}}

\def\F{\mathcal{F}}
\def\G{\mathcal{G}}
\def\geq{\geqslant}

\def\im{\mathrm{i}}

\def\leq{\leqslant}
\def\N{\mathbb{N}}

\def\P{\mathbb{P}}
\def\Pr{\mathbf{P}}

\def\R{\mathbb{R}}

\def\Var{\operatorname{Var}}

\begin{document}

\title{Left tail asymptotics for the density of supercritical branching processes in random environment}
\author[a]{Panyi Chen\thanks{panyichen@mail.bnu.edu.cn}}
\author[b]{Mei Zhang\thanks{Corresponding author: meizhang@bnu.edu.cn}}
\affil[a,b]{Laboratory of Mathematics and Complex Systems (Ministry of Education), School of Mathematical Sciences, Beijing Normal University, Beijing 100875, P. R. China}
\maketitle
\noindent\textit{Abstract:} Let $(Z_n)_{n\geq 0}$ be a supercritical branching process in a random environment in which every individual produces at least one offspring, and let $W$ be the almost sure limit of the normalized $Z_n$. Write $p_1$ for the (random) probability that an individual has exactly one child and $m_0$ for the mean offspring number in a given environment.   Assume that $p_1$ is positive almost surely and bounded away from $1$, that $m_0$ is bounded away from $1$ and from infinity together with a mild moment condition, and that $\log m_0$ is non-lattice.  Let $a_0>0$ be the unique solution of the equation $\E[p_1 m_0^{a_0}]=1$.   Under these assumptions, we determine the precise left tail asymptotics of the annealed law of $W$: the distribution function $F(x)=\P(W\leq x)$ satisfies $F(x)\sim D x^{a_0}$ as $x\downarrow 0$ with a constant $D\in(0,\infty)$ admitting an explicit representation, and the annealed density $w$ of $W$, which we show exists and can be chosen continuous on $(0,\infty)$, satisfies $w(x)\sim G x^{a_0-1}$ as $x\downarrow 0$ with $G=a_0 D>0$. These results extend the classical power-law left tail behavior of the martingale limit of a Galton--Watson process in the literature to the random environment setting, where the exponent $a_0$ is determined by the averaged environmental law rather than by a fixed offspring distribution.
\bigskip

\bigskip


\section{Introduction}

Consider a supercritical Galton--Watson process $\{Z_n^0:n\ge 0\}$. It is well known that the normalized population size $Z_n^0/m^n$ forms a nonnegative martingale converging almost surely to a limit $W^0$, which is strictly positive under certain conditions (see e.g. \cite{AthreyaNey1972}). The tail behavior of $W^0$ has been studied for a long time and exhibits two different regimes. When the minimum family size is at least two, both tails of $W^0$ decay exponentially fast, with rates governed by the maximum and minimum family sizes; this is the main focus of Biggins and Bingham \cite{BigginsBingham1993}, who combined the asymptotics of the moment generating function of $W^0$ with G\"artner--Ellis type large deviation arguments. When the minimum family size is one, the left tail of $W^0$ decays instead like a power: Dubuc \cite{Dubuc1971} proved that the density of $W^0$ behaves like a power of $x$ near the origin, and Biggins and Bingham \cite{BigginsBingham1993} derived a strengthening of Dubuc's result on the behavior of the density near $0$. In this regime the exponent is determined by the probability $p_1$ of having exactly one child through the equation $p_1 m^{a}=1$.


Branching processes in a random environment (BPRE) were introduced by Smith and Wilkinson \cite{smith1969} and have been extensively studied as a natural generalization of the Galton--Watson process in recent years, see \cite{vat2005,2021afa,Athreya1971a,Athreya1971b}, \cite{2013ban,2014ban,2014boi}, \cite{2022dar,2022bur,Damek2019Absolute}, \cite{2017grama,grama2023,Vat2017}, etc. For the supercritical case, Grama, Liu and Miqueu \cite{2017grama} investigated the harmonic moments and large deviations of the BPRE. Damek, Gantert and Kolesko \cite{Damek2019Absolute} established the absolute continuity of the law of the martingale limit $W$. Later, Grama, Liu and Miqueu \cite{grama2023} showed that, under their boundedness assumptions and when each individual has at least one child, the solution $a_0>0$ of $\E[p_1m_0^{a_0}]=1$ is the critical value for the existence of the harmonic moments $\E W^{-b}$.

In the present paper, we study the precise left tail asymptotics of the annealed law of $W$ for a supercritical BPRE with $p_0=0$. Our main result (\cref{mainthm}) states that, under the standing assumptions stated below, including a boundedness assumption on the mean offspring numbers $m_0$ and a non-lattice assumption on $\log m_0$, the distribution function $F$ and the density $w$ of $W$ satisfy
$$
F(x)=\P(W\leq x)\sim D x^{a_0}, \qquad w(x)\sim G x^{a_0-1} \quad \text{as } x\downarrow 0,
$$
where $D\in(0,\infty)$ has the representation \eqref{Dformula} and $G=a_0 D>0$. This result extends the left tail asymptotics of Dubuc \cite{Dubuc1971} and Biggins and Bingham \cite{BigginsBingham1993} from the Galton--Watson process to the BPRE. In contrast with the constant environment case, the exponent $a_0$ is now determined by the averaged equation $\E[p_1 m_0^{a_0}]=1$ involving the whole environmental law.

We note that Kutsenko \cite{Kutsenko2025} derived a complete left-tail expansion with multiplicatively periodic coefficients for BPREs with a common offspring mean across environments, under additional analytic assumptions. Here, the mean varies with the environment and $\log m_0$ is non-lattice; we obtain a leading-order density asymptotic with a constant prefactor by renewal methods.

Our approach is based on the decomposition $W\overset{d}{=}m_0^{-1}\sum_{j=1}^{Z_1}W^{(j)}$, where, conditionally on the environment, the $W^{(j)}$ are independent with the martingale-limit law corresponding to the shifted environment. Starting from it, we get renewal type equations for $F$ and $w$ driven by the associated random walk $S_n=\log\Pi_n$. Meanwhile, using the identity $\E[p_1 m_0^{a_0}]=1$, we establish a Cram\'er-type change of measure on the environment sequence, under which $(S_n)$ becomes a random walk with positive drift $\tilde{\mu}$.  This measure change  together with the key renewal theorem leads to the asymptotics of $F$;  see Lemma \ref{lemma-renewal} and the proof of Theorem \ref{mainthm}. Passing from the distribution function to the density is more delicate, since the asymptotics of $F$ cannot simply be differentiated. The main task is to control the convolution powers $w_\xi^{*j}$ of the quenched density, uniformly in $j$, before averaging over the common environment.   First, in Lemmas \ref{l2.2}--\ref{l2.3} we obtain some estimates of the quenched characteristic functions $\Phi_\xi$. Based on them,  we then get in Lemma \ref{quenchedregularity} the quenched estimates $w_\xi^{*j}(x)\le CB(\xi)x^{b-1}F_\xi(x)$ uniformly in $j\ge2$, which is crucial to the proof of Theorem \ref{mainthm} for the density.

The paper is organized as follows. In the remainder of this section we describe the model, fix the notation and state the standing assumptions and the main result. In \cref{sec:lemmas} we establish the existence and regularity properties of the density of $W$. \cref{sec:main} is devoted to the proof of the main result; an explicit example with linear fractional offspring distributions is given in \cref{rem:LF}.

\medskip

We start by describing a BPRE. The random environment is represented by a sequence $\xi=(\xi_0,\xi_1,\ldots)$ of independent and identically distributed random variables taking values in an abstract space $\Xi$. The random variable $\xi_n$ represents the random environment at time $n$; to each realization of $\xi_n$ corresponds a probability law $\{p_i(\xi_n): i\in \N\}$ on $\N=\{0,1,2,\ldots\}$, whose probability generating function is
$$
f_{\xi_n}(s)=f_n(s)=\sum_{i=0}^\infty p_i(\xi_n)s^i, \quad s\in [0,1], \quad p_i(\xi_n)\geq 0, \quad \sum_{i=0}^\infty p_i(\xi_n)=1.
$$
Define the process $(Z_n)_{n\geq 0}$ by the relations
\begin{align}\label{zn+1}
Z_0=1, \quad Z_{n+1}=\sum_{i=1}^{Z_n}N_{n,i}, \quad n\geq 0,
\end{align}
where $N_{n,i}$ is the number of children of the $i$-th individual of the generation $n$. Conditionally on the environment $\xi$, the r.v.'s $N_{n,i}$ $(n\geq0,\ i=1,2,\ldots)$ are mutually
independent, and for each $n$, the variables $N_{n,i}$ have a common probability generating function $f_n$.

We denote by $\Pr_\xi$ the \textit{quenched law}, i.e. the conditional probability when the environment $\xi$ is given, and by $\nu$ the law of the environment $\xi$. Then $\P(\dif x,\dif \xi)=\Pr_\xi(\dif x)\nu(\dif \xi)$ is the total law, called the \textit{annealed law}. The corresponding quenched and annealed expectations are denoted by $\Ex_\xi$ and $\E$, respectively.

Put
$$
\begin{gathered}
m_n(a)=m_n(\xi,a)=\sum_{i=0}^\infty i^a p_i(\xi_n), \quad m_n=m_n(1),\quad \Pi_n=\Ex_\xi Z_n,
\end{gathered}
$$
where $m_{n}$ represents the average number of children of an individual of generation $n$ when the environment $\xi$ is given. Let $W_n=\dfrac{Z_n}{\Pi_n}$ be the normalized population size. Then $(W_n)_{n\geq 0}$ forms a nonnegative martingale under both the quenched law $\Pr_\xi$ and the annealed law $\P$, with respect to the natural filtration
$$
\F_n=\sigma(\xi, N_{k,i},0\leq k\leq n-1, i=1,2,\ldots),
$$
where by convention $\F_0=\sigma(\xi)$. The limit $W=\lim W_n$ therefore exists $\P$-a.s.

We consider the associated random walk
\begin{equation}\label{Sn0}
S_n=\log \Pi_n=\sum_{i=1}^n \log m_{i-1},
\end{equation}
where $\log m_{i-1}$ are i.i.d. depending only on the environment $\xi$.

For the sake of brevity, we set $\mu:=\E\log m_0$ and $\sigma^2:=\Var(\log m_0)$.
We shall assume that the BPRE is supercritical with $\mu\in(0,\infty)$, and $(S_n)$ is non-degenerate with $0<\sigma^2<\infty$.

Throughout the paper, we assume the following condition:
$$
\E\left[\dfrac{Z_1\log^+ Z_1}{m_0}\right]<\infty,
$$
which implies that $W_n\to W$ in $L^1$, and
$$
\P(W>0)=\P(Z_n\to\infty)=\lim_{n\to\infty}\P(Z_n>0)>0.
$$
See e.g. \cite[Theorem 2]{TANNY1988123}.

Furthermore, we assume in the sequel that each individual has at least one child, which means that
$$
p_0=0, \quad  \P\text{-a.s.}
$$
We also assume that
$$
p_1>0, \quad \P\text{-a.s.}, \qquad\text{and}\qquad
\bar{p}:=\esssup p_1<1.
$$
This implies that the associated random walk has positive increments, $Z_n\to \infty$ and $W>0$ $\P$-a.s.

Then the behavior of $(Z_n)$ is mainly determined by $S_n$ which is seen from the decomposition
$$
\log Z_n=S_n+\log W_n.
$$

As in \cite{grama2023}, we assume

\begin{assumption}\label{boundedm}
There exist constants $\underline{m}$, $\bar{m}$, $\delta>1$ such that
  $$
1<\underline{m}\leq m_0\leq \bar{m}, \quad m_0(\delta)\leq \bar{m}^\delta. \quad \P\text{-a.s}.
$$
\end{assumption}

Under \cref{boundedm}, by \cite[Corollary 2.3]{grama2023}, the equation $\E[p_1 m_0^{a}]=1$ has a unique solution $a=a_0>0$, which is the critical value for the existence of the harmonic moments of $W$.

We also assume
\begin{assumption}\label{nonlattice}
The random variable $\log m_0$ is non-lattice.
\end{assumption}

The following is the main result of this paper:

\begin{theorem}\label{mainthm}
Suppose \cref{boundedm,nonlattice} hold. Let $a_0$ be the unique solution of $\E[p_1 m_0^{a_0}]=1$. Define $F(x):=\P(W\leq x)$ and $\mu_W(\cdot)=\P(W\in \cdot)$. Then
$\mu_W$ is absolutely continuous w.r.t. Lebesgue measure and the density $w$ can be chosen to be continuous on $(0,\infty)$. Furthermore, there exist constants $D>0$ and $G=a_0 D$ such that
\begin{equation}\label{main00}
w(x)\sim G x^{a_0-1}, \qquad \mbox{and}\qquad  F(x)\sim D x^{a_0}, \qquad x\downarrow 0.
\end{equation}
\end{theorem}

We close this section by fixing some notation. We denote by $\Phi(t):=\E[\e^{\im tW}]$ and $\Phi_\xi(t):=\Ex_\xi[\e^{\im tW}]$ the annealed and quenched characteristic functions, and put
$$
F_\xi(x):=\Pr_\xi(W\leq x),\qquad
F_j(x):=\E[F_\xi^{*j}(x)].
$$
Here $F_j$ is the distribution function of the martingale limit for a process $Z_n$ started with $j$ individuals in one common environment, normalized by $\Pi_n$. We write $\P_j$ and $\E_j$ for this annealed law and its expectation, and $w$ and $w_j$ for the densities of $F$ and $F_j$, respectively. We use $T\xi=(\xi_1,\xi_2,\ldots)$ for the environment shift.

$W$ admits the decomposition
\begin{equation}\label{decomposition}
W\overset{d}{=}\dfrac{1}{m_0}\sum_{j=1}^N  W^{(j)},
\end{equation}
where $N=Z_1=N_{0,1}$ and, conditionally on $\xi$, the variables $W^{(j)}$ are independent of each other and of $N$, with common law $\Pr_{T\xi}(W\in\cdot)$. Under $\P$, the vector $(W^{(j)})_{j\geq1}$ is independent of $(N,\xi_0)$, but its coordinates share the future environment and are generally dependent. In this decomposition and its applications, $W^{(j)}$ is the martingale limit of a subtree.

Throughout the paper, $C$ denotes a positive constant whose value may differ from line to line, and $\1_A$ denotes the indicator function of a set $A$. For positive functions $f$ and $g$, we write $f(x)\sim g(x)$ as $x\downarrow0$ if $f(x)/g(x)\to1$ as $x\downarrow0$. For random variables $(Y_n)_{n\geq0}$ and $Y$, we write $Y_n\Rightarrow Y$ for convergence in distribution. Finally, $K\Subset(0,\infty)$ means that $K$ is compactly contained in $(0,\infty)$, i.e.\ $K$ is compact and $K\subset(0,\infty)$.


\section{Auxiliary Lemmas}\label{sec:lemmas}

In this section, we establish the existence and regularity properties of the density of $W$, which is key to the proof of the main result.

  All quenched assertions below are understood on a common set of full $\nu$-measure, chosen to be stable under $T$. Suprema over environments are taken over this set.

  \begin{lemma}\label{l2.1}
  Fix $1<p\leq2$.   There exists $C_p>0$, depending only on $p$, such that for every $k\geq1$ and i.i.d. nonnegative random variables $X_1,\ldots,X_k$ with mean $m$ and $p$-th
moment $m(p)<\infty$, it holds that
\begin{equation}\label{vBEmean}
\E\Big(\frac1k\sum_{j=1}^kX_j\Big)^p\leq m^p+C_pk^{1-p}m(p).
\end{equation}
\end{lemma}

\begin{proof}
Writing $S_k=\sum_{j=1}^k(X_j-m)$, the linear term
$pm^{p-1}\E(k^{-1}S_k)$ vanishes, and Theorem~2 of
\cite{vonBahrEsseen1965} yields
\begin{align*}
\E\Big(\frac1k\sum_{j=1}^kX_j\Big)^p
&=\E\big(m+k^{-1}S_k\big)^p
\leq m^p+C_pk^{-p}\E|S_k|^p
\leq m^p+2C_pk^{1-p}\E|X_1-m|^p\\
&
\leq m^p+2^{p+1}C_pk^{1-p}m(p),
\end{align*}
where $\E|X_1-m|^p\leq2^{p-1}(m(p)+m^p)\leq2^pm(p)$.

\end{proof}


\begin{lemma}\label{pmoment}
Fix $p\in(1,\min\{\delta,2\}]$. Then we have
\begin{equation}\label{uniformpositive}
\Ex_\xi W_n=\Ex_\xi W=1, \quad  \sup_\xi\Ex_\xi W^p\leq K<\infty.
\end{equation}
\end{lemma}


\begin{proof}

By \eqref{zn+1}, $$W_{n+1}=\frac{W_n}{m_n}\cdot Z_n^{-1}\sum_{j=1}^{Z_n}N_{n,j}, $$
where  under $\Pr_\xi$, conditionally on $Z_n$, the variables $\{N_{n,j}:1\leq j\leq Z_n\}$
are i.i.d. with mean
$m_n$ and $p$-th moment $m_n(p)$. By this decomposition we can obtain $\Ex_\xi W_n=1$. Applying \eqref{vBEmean} and using
$Z_n^{1-p}W_n^p=\Pi_n^{1-p}W_n$, we obtain
$$
\Ex_\xi W_{n+1}^p
\leq\Ex_\xi W_n^p+C_p\,\Pi_n^{1-p}\frac{m_n(p)}{m_n^p},
$$
which yields that
$$
\sup_n\Ex_\xi W_n^p
\leq1+C_p\sum_{n\geq0}\Pi_n^{1-p}\frac{m_n(p)}{m_n^p}
\leq1+C_p\left(\frac{\bar m}{\underline m}\right)^p
\sum_{n\geq0}\underline m^{-n(p-1)}=:K<\infty.
$$
Uniform integrability and Fatou's lemma yield \eqref{uniformpositive}.
\end{proof}

For the quenched $L^p$ bound \eqref{uniformpositive}, a similar argument appears in the proof of \cite[Theorem~1.1]{HuangLiu2014}. Other related moment results can be found in \cite{Liu2001,HuangLiu2012,LiangLiu2013}.


\begin{lemma}\label{l2.4}
  Let $v$ be a probability density on $[0,\infty)$ with distribution
function $V$ such that
$$
v(x)\leq A_0x^{d-1},\qquad 0<x\leq1
$$
for some  $A_0,d\geq0$. Then, for each fixed $k\geq2$,  there exists $C(k,d)>0$, depending only on $k$ and $d$, such that
\begin{equation}\label{convolutionestimatek}
v^{*k}(x)\leq C(k,d) A_0x^{d-1}V(x)^{k-1},\qquad 0<x\leq1.
\end{equation}
  If in addition,
$$
V(x)\leq K_0x^\alpha,\qquad 0<x\leq1
$$
  for some $K_0,\alpha>0$, then there exists a constant $C(K_0,\alpha,d)$, depending only on $K_0,\alpha,d$, such that
\begin{equation}\label{convolutionestimate}
v^{*j}(x)\leq C(K_0,\alpha,d) A_0x^{d-1}V(x),\qquad j\geq2,\quad 0<x\leq1.
\end{equation}
\end{lemma}

\begin{proof}
 We first prove \eqref{convolutionestimatek} by induction on $k$.
  For $k=1$ it
is the assumption $v(x)\leq A_0x^{d-1}$, with $C(1,d)=1$. Assume
\eqref{convolutionestimatek} holds for $k-1$. Since $u^{d-1}$ is
decreasing in $u$ if $d<1$ and increasing if $d\geq1$, we have
\begin{equation}\label{halfsplit}
u^{d-1}\leq2^{|d-1|}x^{d-1},
\qquad u\in[x/2,x].
\end{equation}
Splitting the convolution at $x/2$,
\begin{equation}\label{convsplit}
v^{*k}(x)
=\Big(\int_0^{x/2}+\int_{x/2}^x\Big)v^{*(k-1)}(x-y)v(y)\dif y.
\end{equation}
For the first integral, the induction hypothesis together with
\eqref{halfsplit} (applied to $u=x-y\in[x/2,x]$) and the monotonicity
of $V$ give
\begin{align*}
\int_0^{x/2}v^{*(k-1)}(x-y)v(y)\dif y
&\leq C(k-1,d)A_0\,2^{|d-1|}x^{d-1}V(x)^{k-2}\int_0^{x/2}v(y)\dif y\\
&\leq C(k-1,d)A_0\,2^{|d-1|}x^{d-1}V(x)^{k-1}.
\end{align*}
For the second integral, \eqref{halfsplit} applied to $u=y$ gives
$v(y)\leq A_0y^{d-1}\leq A_0\,2^{|d-1|}x^{d-1}$, and substituting
$u=x-y$ we obtain
\begin{align*}
\int_{x/2}^xv^{*(k-1)}(x-y)v(y)\dif y
&\leq A_0\,2^{|d-1|}x^{d-1}\int_0^{x/2}v^{*(k-1)}(u)\dif u\\
&\leq A_0\,2^{|d-1|}x^{d-1}V^{*(k-1)}(x)
\leq A_0\,2^{|d-1|}x^{d-1}V(x)^{k-1},
\end{align*}
where in the last step we note that $V^{*(k-1)}(x)\leq V(x)^{k-1}$.  Adding the two bounds proves \eqref{convolutionestimatek}
with $C(k,d)=2^{|d-1|}(C(k-1,d)+1)$.

We now derive \eqref{convolutionestimate}. Choose $k_0\geq2$ with
$d-1+\alpha(k_0-2)\geq0$. Since $V(x)^{k_0-2}\leq(K_0x^\alpha)^{k_0-2}$,
the bound \eqref{convolutionestimatek} for $k_0$ reads
$$
v^{*k_0}(x)\leq C(k_0,d)K_0^{k_0-2}A_0x^{c}V(x),
\qquad c:=d-1+\alpha(k_0-2)\geq0.
$$
Since $x^{c}V(x)$ is increasing in $x$, for $j>k_0$,
$$
\begin{aligned}
v^{*j}(x)
&=\int_0^xv^{*k_0}(x-y)v^{*(j-k_0)}(y)\dif y\\
&\leq C(k_0,d)K_0^{k_0-2}A_0x^{c}V(x)\int_0^xv^{*(j-k_0)}(y)\dif y\\
&\leq C(k_0,d)K_0^{k_0-2}A_0x^{c}V(x).
\end{aligned}
$$
Since $c\geq d-1$ and $x\leq1$, we have $x^c\leq x^{d-1}$, which proves
\eqref{convolutionestimate} for all $j\geq k_0$. For  $j\in\{2,\ldots,k_0-1\}$, \eqref{convolutionestimatek} and
$V(x)\leq1$ give $v^{*j}(x)\leq C_jA_0x^{d-1}V(x)$, and enlarging $C(K_0,\alpha,d)$
to $\max\{C(k_0,d)K_0^{k_0-2},C_2,\ldots,C_{k_0-1}\}$ concludes the proof.
\end{proof}


In the following we introduce a measure change. Since $\E[p_1 m_0^{a_0}]=1$, we may tilt each marginal of the environment by the weight $p_1m_0^{a_0}$: let $\tilde{\P}$ be the probability measure, with expectation $\tilde{\E}$, under which $\xi=(\xi_0,\xi_1,\ldots)$ remains i.i.d.\ and
\begin{equation}\label{tilted-measure}
\dfrac{\dif \tilde{\P}}{\dif \P}\Bigg{|}_{\sigma(\xi_0,\ldots,\xi_{n-1})}=\prod_{i=0}^{n-1}p_1(\xi_i) m_i^{a_0},
\qquad n\geq 1.
\end{equation}
Under $\tilde{\P}$ the increments $\log m_0,\log m_1,\ldots$ of the random walk $(S_n)$ in \eqref{Sn0} are i.i.d., non-lattice, take values in $[\log\underline{m},\log\bar{m}]$, and have mean
$$
\tilde{\mu}:=\tilde{\E}\log m_0=\E\big[p_1 m_0^{a_0}\log m_0\big]\in[\log\underline{m},\log\bar{m}]\subset(0,\infty).
$$
Let $\tilde{U}(B):=\sum_{k\geq 0}\tilde{\P}(S_k\in B)$ be the renewal measure of $(S_k)$ under $\tilde{\P}$. Since $\log m_i\geq\log\underline{m}>0$ a.s.,
\begin{equation}\label{renewal0}
\tilde{U}([s,s+L))\leq 1+\dfrac{L}{\log \underline{m}},
\qquad \forall s\geq 0,\ L>0.
\end{equation}


By renewal theory, we obtain

\begin{lemma}\label{lemma-renewal}
Suppose that $H$ is a nonnegative directly Riemann integrable (DRI) function such that \begin{equation}\label{Ht00}
H(t)\leq C\e^{-c|t|},\qquad t\in\R.
\end{equation}
Then

\begin{equation}\label{renewal00}
\int_0^t H(t-s)\,\tilde{U}(\dif s)\longrightarrow \dfrac{1}{\tilde{\mu}}\int_0^\infty H(s)\dif s,
\qquad t\to\infty.
\end{equation}

\begin{equation}\label{tail}
\int_t^\infty H(t-s)\,\tilde{U}(\dif s)\longrightarrow \frac{1}{\tilde{\mu}}\int_{-\infty}^0 H(s)\,\dif s,
\qquad t\to\infty.
\end{equation}
\end{lemma}


\begin{proof}

Applying the key renewal theorem \cite[Theorem 3.4.2]{Ross1996} to the DRI function $H\1_{\{t\geq0\}}$ gives \eqref{renewal00}.

For \eqref{tail},
 set $u=s-t$ and $g(u):=H(-u)\1_{\{u\geq0\}}$; then
$$
\int_t^\infty H(t-s)\,\tilde{U}(\dif s)=\int_0^\infty g(u)\,\tilde{U}(t+\dif u),
\qquad
\int_0^\infty g(u)\,\dif u=\int_{-\infty}^0 H(s)\,\dif s,
$$
where $\tilde U(t+\dif u)$ denotes the shifted renewal measure. By \eqref{Ht00}, for $u>0$, we have $|g(u)|\leq C\e^{-cu}$, so $g$ is DRI. Let $g_{(n)}\leq g\leq g^{(n)}$ be its dyadic lower and upper step functions, i.e.
$$
g_{(n)}(u):=\sum_{k\geq0}\Big(\inf_{v\in[k2^{-n},\,(k+1)2^{-n})}g(v)\Big)\,\1_{[k2^{-n},\,(k+1)2^{-n})}(u),
$$
$$
g^{(n)}(u):=\sum_{k\geq0}\Big(\sup_{v\in[k2^{-n},\,(k+1)2^{-n})}g(v)\Big)\,\1_{[k2^{-n},\,(k+1)2^{-n})}(u).
$$

Since $g$ is continuous a.e.\ and bounded, it is Riemann integrable on $[0,K]$ for every $K>0$, hence $\int_0^K(g^{(n)}-g_{(n)})\dif u\to0$. Moreover, taking the supremum of the exponential upper bound on each dyadic interval gives
$$
g^{(n)}(u)\leq C\sum_{k\geq0}\e^{-ck2^{-n}}\,\1_{[k2^{-n},\,(k+1)2^{-n})}(u) \leq C\e^{-c(u-2^{-n})}\leq C\e^{c}\e^{-cu}
\qquad(\text{since }2^{-n}\leq1),
$$
hence $g^{(n)}-g_{(n)}\leq g^{(n)}\leq C\e^{c}\e^{-cu}$ uniformly in $n$ and the tails are uniformly small. Therefore $\int_0^\infty(g^{(n)}-g_{(n)})\dif u\to0$ as $n\to\infty$. By Blackwell's renewal theorem \cite[Theorem 3.4.1]{Ross1996}, for every $0\leq a<b$,
\begin{equation}\label{blackwell}
\tilde U([t+a,t+b))\longrightarrow\frac{b-a}{\tilde\mu},
\qquad t\to\infty.
\end{equation}
We claim that, for each fixed $n$ and each
$\phi\in\{g_{(n)},g^{(n)}\}$,
\begin{equation}\label{stepconvergence}
\int_0^\infty\phi(u)\,\tilde U(t+\dif u)
\longrightarrow\frac1{\tilde\mu}\int_0^\infty\phi(u)\,\dif u,
\qquad t\to\infty.
\end{equation}
For fixed $K>0$, setting $\phi_K:=\phi\,\1_{[0,K)}$, which is a finite sum
$\phi_K=\sum_{i=1}^{N_K}c_i\1_{[a_i,b_i)}$ of interval indicators,
applying \eqref{blackwell} leads to
\begin{equation}\label{truncatedconvergence}
\int_0^\infty\phi_K(u)\,\tilde U(t+\dif u)
=\sum_{i=1}^{N_K}c_i\,\tilde U([t+a_i,t+b_i))
\longrightarrow\frac1{\tilde\mu}\sum_{i=1}^{N_K}c_i(b_i-a_i)
=\frac1{\tilde\mu}\int_0^\infty\phi_K(u)\,\dif u.
\end{equation}
The tails are small uniformly in $t\geq0$ and $n$: recalling
$0\leq g_{(n)}(u)\leq g^{(n)}(u)\leq\bar C\e^{-cu}$ with $\bar C:=C\e^{c}$,
splitting into unit intervals and applying \eqref{renewal0},
\begin{equation}\label{renewaltail}
\begin{aligned}
\int_K^\infty\phi(u)\,\tilde U(t+\dif u)
&\leq\bar C\sum_{j\geq\lfloor K\rfloor}\e^{-cj}\,
\tilde U([t+j,t+j+1))\\
&\leq\bar C\Big(1+\frac1{\log\underline m}\Big)
\sum_{j\geq\lfloor K\rfloor}\e^{-cj}
\leq\bar C\Big(1+\frac1{\log\underline m}\Big)
\frac{\e^{-c(K-1)}}{1-\e^{-c}},
\end{aligned}
\end{equation}
while
\begin{equation}\label{lebesguetail}
\int_K^\infty\phi(u)\,\dif u
\leq\bar C\int_K^\infty\e^{-cu}\dif u
=\frac{\bar C}{c}\,\e^{-cK}.
\end{equation}
Given $\varepsilon>0$, choose $K=K(\varepsilon)$ so large that the
bounds in \eqref{renewaltail}--\eqref{lebesguetail} are both
$<\varepsilon$; then \eqref{truncatedconvergence} yields
$$
\limsup_{t\to\infty}\Big|
\int_0^\infty\phi(u)\,\tilde U(t+\dif u)
-\frac1{\tilde\mu}\int_0^\infty\phi(u)\,\dif u\Big|
\leq(2+\tilde\mu^{-1})\varepsilon,
$$
and since $\varepsilon>0$ is arbitrary, \eqref{stepconvergence}
follows. Finally, $\tilde U(\{t\})\leq\tilde U([t,t+h))$ for every
$h>0$; by \eqref{blackwell} we have
$$
\limsup_{t\to\infty}\tilde U(\{t\})
\leq\frac{h}{\tilde\mu}
\qquad h>0.
$$
Letting $h\downarrow0$ gives $\tilde U(\{t\})\to0$.
Consequently the endpoint in the split at $s=t$ contributes at most
$\|H\|_\infty\,\tilde U(\{t\})\to0$ and does not affect the limit.

Applying this to $g_{(n)}$ and $g^{(n)}$ gives
$$
\frac{1}{\tilde\mu}\int_0^\infty g_{(n)}\,\dif u
\;\leq\;\liminf_{t\to\infty}\int_0^\infty g\,\tilde U(t+\dif u)
\;\leq\;\limsup_{t\to\infty}\int_0^\infty g\,\tilde U(t+\dif u)
\;\leq\;\frac{1}{\tilde\mu}\int_0^\infty g^{(n)}\,\dif u,
$$
and letting $n\to\infty$ proves \eqref{tail}.
\end{proof}


 \begin{lemma}\label{l2.2}
There  exist deterministic constants $t_0\in(0,1)$ and $c_0>0$ such that for almost every $\xi$,
\begin{equation}\label{goodfrequency}
|\Phi_\xi(t)|\leq1-c_0 p_1(\xi_0) t^2,
\qquad |t|\leq t_0.
\end{equation}

\end{lemma}

\begin{proof}

Choose $u,v$ such that $\underline{m}^{-1}<u<v<1$. From \eqref{decomposition}, noting that $\Ex_\xi W^{(1)}=\Ex_{T\xi}W=1$ and $W^{(1)}$ is independent of $Z_1$ and $m_0$ under $\Pr_\xi$, Markov's inequality gives
\begin{align*}
\Pr_\xi(W\leq u)&\geq \Pr_\xi(Z_1=1, m_0^{-1}W^{(1)}\leq u)\\
&\geq p_1(1-\Pr_\xi(W^{(1)}> m_0 u))\geq p_1\left(1-\frac1{m_0u}\right)
\geq p_1\left(1-\frac1{\underline m u}\right).
\end{align*}
Also, by Lemma \ref{pmoment} and H\"older's inequality,
$$
1=\Ex_\xi W\leq v+\Ex_\xi (W\1_{\{W>v\}})\leq v+K^{1/p}\Pr_\xi(W>v)^{1-1/p}.
$$
Consequently there exist deterministic $R>1$ and $d>0$ such that $\Pr_\xi(v<W\leq R)\geq d$ for all such environments. Let $W'$ be a conditionally independent copy of $W$ under $\Pr_\xi$. On the event $A:=\{W\leq u,\ v<W'\leq R\}$, the difference $W'-W$ belongs to $[v-u,R]$.

For $|t|\leq1/R$, we use  $1-\cos z\geq z^2/4$ for $|z|\leq1$ in
$$
1-|\Phi_\xi(t)|^2=\Ex_\xi[1-\cos(t(W-W'))]\geq \Ex_\xi[(1-\cos(t(W-W')))\1_A]\geq \frac{(v-u)^2t^2}{4}\Pr_\xi(A).
$$
By the preceding bounds, we have
$$
\Pr_\xi(A)\, \geq \,dp_1\left(1-\frac1{\underline m u}\right).
$$
Then for $|t|\leq1/R$,
\begin{align*}
|\Phi_\xi(t)|
 \leq\sqrt{1- \frac{(v-u)^2t^2p_1\left(1-\frac1{\underline m u}\right)d}{4}}
 \leq 1-\frac{(v-u)^2t^2\left(1-\frac1{\underline m u}\right)dp_1}{8}.
\end{align*}
Choosing $t_0=1/R$ and $c_0=\frac{1}{8} (v-u)^2\left(1-\frac1{\underline m u}\right)d$, we complete the proof.
\end{proof}


\begin{lemma}\label{l2.3n}
Put $g(s)=\bar p s+(1-\bar p)s^2$ and fix $\rho\in(\bar p,1)$. Let $g_k$ denote the $k$-fold iterate of $g$. Then there exist $C_0,d_0>0$ such that
\begin{equation}\label{giteration}
g_k(1-z)\leq C_0\rho^kz^{-d_0},\qquad k\geq0,\quad 0<z\leq1.
\end{equation}

\end{lemma}

\begin{proof}

 To verify this, note first that $g(0)=0$, $g(1)=1$ and
$$
g(s)-s=-(1-\bar p)s(1-s)<0,\qquad 0<s<1,
$$
so every orbit $s_k:=g_k(1-z)$ is  nonincreasing and converges to $0$. We treat the two phases of the iteration separately.

Choose
$$
s_*:=\frac{\rho-\bar p}{1-\bar p}\in(0,1),
\qquad\text{so that}\qquad
\bar p+(1-\bar p)s_*=\rho.
$$
Since $g(s)/s=\bar p+(1-\bar p)s$ is nondecreasing in $s$, we have $g(s)\leq\rho s$ for $0\leq s\leq s_*$. Hence,
\begin{equation}\label{Contractionzone}
g_j(1-z)\leq s_*
\quad\Longrightarrow\quad
g_{j+i}(1-z)\leq\rho^{i}s_*,\qquad i\geq0.
\end{equation}
Since
$$
1-g(s)=(1-s)\big[1+(1-\bar p)s\big],
$$
for $s_*\leq s<1$,
\begin{equation}\label{Escapezone}
1-g(s)\geq\lambda(1-s),
\qquad \lambda:=1+(1-\bar p)s_*>1.
\end{equation}
  Set $k_1(z):=\inf\{j\geq0:g_j(1-z)\leq s_*\}$.

If $1-z\leq s_*$, then $k_1(z)=0$.

If $1-z>s_*$, applying \eqref{Escapezone} at  $j<k_1(z)$ gives $1-g_j(1-z)\geq\lambda^{j}z$. Choosing $j=k_1(z)-1$, by the definition of $k_1$,  $g_j(1-z)>s_*$, and then  $\lambda^{k_1(z)-1}z<1-s_*$. Since $\log(1-s_*)<0$,
\begin{equation}\label{k1z}
k_1(z)\leq 1+\frac{\log\big((1-s_*)/z\big)}{\log\lambda}
\leq C+\frac{\log(1/z)}{\log\lambda},
\qquad C:=1.
\end{equation}

Fix $z\in(0,1]$ and put
$
d_0:=\frac{\log(1/\rho)}{\log\lambda},$
so that
$$\rho^{-\log(1/z)/\log\lambda}=\e^{d_0\log(1/z)}=z^{-d_0}.
$$
If $k\geq k_1(z)$, \eqref{Contractionzone} together with the above bound on $k_1(z)$ yields that
$$
g_k(1-z)\leq\rho^{k-k_1(z)}s_*
\leq\rho^{-C}\rho^{k}z^{-d_0}.
$$
If $k<k_1(z)$, by \eqref{k1z} and $\rho\in (0,1)$ we have $\rho^{k}>\rho^{C}z^{d_0}$, hence $$ \rho^{-C}\rho^{k}z^{-d_0}>1\geq g_k(1-z).$$ Taking $C_0:=\rho^{-C}$,  we obtain \eqref{giteration} with $d_0=\log(1/\rho)/\log\lambda$.
\end{proof}


\begin{lemma}\label{l2.3}
Fix $\rho\in(\bar p,1)$. Let $\beta:=\log(1/\rho)/(2\log\bar m)$. Then
\begin{equation}\label{Dbetabound}
D_\beta(\xi)
:=\sup_{t\in\R}(1+|t|)^\beta|\Phi_\xi(t)|
<\infty, \quad a.s.,
\end{equation}
and for any $r\geq1$, it holds that
\begin{equation}\label{fourierconstant}
 \E {  [D_\beta(\xi)^r]}<\infty.
\end{equation}
\end{lemma}

\begin{proof}

Let $d_0=\log(1/\rho)/\log\lambda$, $M=\max\{1,\rho\bar m^{2d_0}\}$ and   put
$$
q_\varepsilon:=\P(p_1<\varepsilon).
$$
Fix $r\ge1$. Since $p_1>0$ a.s., $q_\varepsilon\downarrow0$ as $\varepsilon\downarrow0$.
  Choose $\varepsilon\in(0,1)$ so small that $q_\varepsilon M^r<1$,  and let
$$
L_n:=\inf\{\ell\geq0:p_1(\xi_{n+\ell})\geq\varepsilon\}.
$$
Then $L_n<\infty$ a.s. and, for each deterministic $n\geq0$,
$\P(L_n=\ell)=(1-q_\varepsilon)q_\varepsilon^\ell$.
For $|t|\geq t_0$, set $n=\inf\{k:\Pi_k\geq |t|/t_0\}$. Then
\begin{equation}\label{Pinbound}
\frac{t_0}{\bar m^{L_n+1}}\leq\frac{|t|}{\Pi_{n+L_n}}\leq t_0.
\end{equation}

Since $p_0(\xi_i)=0$ and $s^j\leq s^2$ for $j\geq2$ and $s\in[0,1]$, every
offspring generating function is dominated by $g$:
\begin{equation}\label{fdomg}
f_i(s)=p_1(\xi_i)s+\sum_{j\ge2}p_j(\xi_i)s^j
\leq p_1(\xi_i)s+\big(1-p_1(\xi_i)\big)s^2
\leq \bar p s+(1-\bar p)s^2=g(s).
\end{equation}
Iterating the quenched branching identity
$\Phi_\xi(t)=f_0\big(\Phi_{T\xi}(t/m_0)\big)$ for $k:=n+L_n$ steps, and
using the monotonicity of each $f_i$ on $[0,1]$ together with
$|\Phi|\le1$ and \eqref{fdomg}, we obtain
\begin{equation}\label{Phiiterate}
|\Phi_\xi(t)|
\leq f_0\circ f_1\circ\cdots\circ f_{k-1}
\Big(\big|\Phi_{T^{k}\xi}\big(t/\Pi_{k}\big)\big|\Big)
\leq g_k\Big(\big|\Phi_{T^{k}\xi}\big(t/\Pi_{k}\big)\big|\Big).
\end{equation}
By the definition of $L_n$, the terminal environment is good:
$p_1(\xi_{n+L_n})\ge\varepsilon$; and the right half of \eqref{Pinbound}
gives $|t|/\Pi_{n+L_n}\leq t_0$. Hence, applying Lemma \ref{l2.2} at
the environment $T^{n+L_n}\xi$, and using the left half of \eqref{Pinbound}
 yields
\begin{equation}\label{Phitail}
\Big|\Phi_{T^{n+L_n}\xi}\Big(\frac{t}{\Pi_{n+L_n}}\Big)\Big|
\leq 1-c_0\varepsilon\Big(\frac{t}{\Pi_{n+L_n}}\Big)^2
\leq 1-z,
\qquad
z:=c_0\varepsilon t_0^2\,\bar m^{-2(L_n+1)},
\end{equation}
where $z\le1$ if we choose  $c_0$ sufficiently small.
Inserting \eqref{Phitail} into \eqref{Phiiterate} and applying Lemma \ref{l2.3n}
with $k=n+L_n$, we get
\begin{align*}
|\Phi_\xi(t)|
\leq g_{n+L_n}(1-z)
\leq C_0\rho^{\,n+L_n}z^{-d_0}
=C_0\big(c_0\varepsilon t_0^2\big)^{-d_0}\bar m^{2d_0}\;
\rho^{\,n}\big(\rho\bar m^{2d_0}\big)^{L_n}.
\end{align*}
  Put
$$
C_\varepsilon:=C_0\big(c_0\varepsilon t_0^2\big)^{-d_0}\bar m^{2d_0}.
$$
  Then for all $|t|\ge t_0$ and
$n=\inf\{k:\Pi_k\geq|t|/t_0\}$,
\begin{equation}\label{Phixit}
|\Phi_\xi(t)|\le C_\varepsilon\rho^n M^{L_n}.
\end{equation}

Define
$$
K_\varepsilon(\xi):=\sup_{n\ge0}\rho^{n/2}M^{L_n}.
$$
For nonnegative
random variables, $\big(\sup_n X_n\big)^r=\sup_n X_n^r\le\sum_n X_n^r$.
Applying this with $X_n=\rho^{n/2}M^{L_n}$ and taking expectations, we have
\begin{equation}\label{Kmomentstep}
\E K_\varepsilon^r
\le\sum_{n\ge0}\rho^{nr/2}\,\E M^{rL_n}.
\end{equation}
Recalling that
$\P(L_n=\ell)=(1-q_\varepsilon)q_\varepsilon^\ell$, we obtain
\begin{equation}\label{geomoment}
\E M^{rL_n}
=\sum_{\ell\ge0}M^{r\ell}(1-q_\varepsilon)q_\varepsilon^\ell
=\frac{1-q_\varepsilon}{1-q_\varepsilon M^r}<\infty
\qquad\text{uniformly in }n.
\end{equation}
Inserting \eqref{geomoment} into \eqref{Kmomentstep} and noting that $\rho^{r/2}<1$,
\begin{equation}\label{Kmoment}
\E K_\varepsilon^r
\le\frac{1-q_\varepsilon}{(1-q_\varepsilon M^r)(1-\rho^{r/2})}<\infty.
\end{equation}
Recall that
$n=\inf\{k:\Pi_k\ge|t|/t_0\}$ and $\Pi_k\le\bar m^k$, hence
\begin{equation}\label{nlowerbound}
n\ge\frac{\log(|t|/t_0)}{\log\bar m},
\qquad |t|\ge t_0.
\end{equation}
Then $$
\rho^{n/2}
=\exp\Big(-\tfrac{n}{2}\log(1/\rho)\Big)
\le\exp\Big(-\frac{\log(1/\rho)}{2\log\bar m}\log(|t|/t_0)\Big)
=\Big(\frac{|t|}{t_0}\Big)^{-\beta},
$$
where $\beta=\log(1/\rho)/(2\log\bar m)>0$. Consequently, by \eqref{Phixit}, for $|t|\ge t_0$,
\begin{equation}\label{Fourierdecay}
|\Phi_\xi(t)|
\le C_\varepsilon K_\varepsilon(\xi)\Big(\frac{|t|}{t_0}\Big)^{-\beta}
=C_\varepsilon t_0^\beta K_\varepsilon(\xi)\,|t|^{-\beta}.
\end{equation}
Multiplying \eqref{Fourierdecay} by $(1+|t|)^\beta$ and using
$(1+|t|)^\beta|t|^{-\beta}\le(1+t_0^{-1})^\beta$ for
$|t|\ge t_0$, while $|\Phi_\xi|\le1$ for $|t|<t_0$, we obtain
\begin{equation}\label{Dbetaboundproof}
D_\beta(\xi)
=\sup_{t\in\R}(1+|t|)^\beta|\Phi_\xi(t)|
\le(1+t_0^{-1})^\beta\big(1+C_\varepsilon K_\varepsilon(\xi)\big).
\end{equation}
Combining \eqref{Dbetaboundproof} with \eqref{Kmoment} and using the elementary
inequality $(a+b)^r\le2^{r-1}(a^r+b^r)$, we arrive at \eqref{fourierconstant}.
\end{proof}

\begin{lemma}[Quenched density estimates]\label{quenchedregularity}
Under \cref{boundedm}, we have

(i) Define
\begin{equation}\label{Adefinition}
A(\xi)=\frac{1}{2\pi}\int_\R|\Phi'_\xi(t)|\dif t.
\end{equation}
Then $\E[A(\xi)^r]<\infty$ for every $r\geq1$.

(ii)  For  almost every $\xi$, the quenched law has a continuous density $w_\xi$ on $(0,\infty)$,  and $$\sup_{x>0}xw_\xi(x)\leq A(\xi).$$

(iii)
There exist deterministic constants $b,C>0$ such that for almost every $\xi$,
\begin{equation}\label{quenchedbounds}
F_\xi(x)\leq Cx^{2b},\qquad 0<x\leq1.
\end{equation}

(iv) There exists a nonnegative random variable $ B(\xi)$ with \begin{equation}\label{regularitymoments0}
\E[B(\xi)^r]<\infty,\qquad r\geq1,
\end{equation}
such that for almost every $\xi$,
\begin{equation}\label{quenchedbounds0}
w_\xi(x)\leq B(\xi)x^{b-1},\qquad 0<x\leq1;
\end{equation}
and uniformly in $j\geq2$ and $0<x\leq1$, we have
\begin{equation}\label{conditionalconvolution}
w_\xi^{*j}(x)\leq CB(\xi)x^{b-1}F_\xi(x).
\end{equation}

\end{lemma}

\begin{proof}

(i) Fix $\rho\in(\bar p,1)$ and let $\beta,g,D_\beta$ be as in \cref{l2.3}.
Differentiating the quenched branching identity
$\Phi_\xi(t)=f_0\big(\Phi_{T\xi}(t/m_0)\big)$ gives
\begin{equation}\label{Phiderivative}
\Phi'_\xi(t)=\frac1{m_0}f'_0\big(\Phi_{T\xi}(t/m_0)\big)\,
\Phi'_{T\xi}(t/m_0).
\end{equation}
First, by \eqref{uniformpositive},
\begin{equation}\label{Phiprimetrivial}
|\Phi'_\xi(t)|
=\big|\im\,\Ex_\xi\big[W\e^{\im tW}\big]\big|
\le\Ex_\xi W=1.
\end{equation}
Second, since $f_0$  has nonnegative coefficients,
$|f'_0(\Phi_{T\xi}(s))|\le f'_0(|\Phi_{T\xi}(s)|)$; and since
$s^{j-1}\le s$ for $j\ge2$ and $s\in[0,1]$,
\begin{equation}\label{fprimebound}
f'_0(s)=p_1+\sum_{j\ge2}jp_js^{j-1}
\le p_1+s\sum_{j\ge2}jp_j
\le p_1+m_0s
\le\bar p+\bar ms.
\end{equation}
Choose $\eta>0$ with $\theta_0:=\bar p+\bar m\eta<1$, and put
$$
C(\xi):=\max\big\{1,\big(D_\beta(\xi)/\eta\big)^{1/\beta}\big\},
\qquad
I_\xi(M):=\int_0^M|\Phi'_\xi(t)|\dif t,
$$
where $D_\beta(\xi)$ is defined by \eqref{Dbetabound}. Then
\begin{equation}\label{Phipsileeta}
|\Phi_\xi(s)|
\le\frac{D_\beta(\xi)}{(1+s)^\beta}
\le D_\beta(\xi)s^{-\beta}
\le\eta,
\qquad s\ge C(\xi).
\end{equation}

We now claim that
\begin{equation}\label{Irecursion}
I_\xi(M)\le\bar m C(T\xi)+\theta_0I_{T\xi}(M).
\end{equation}
  If $M\leq\bar mC(T\xi)$, then
$I_\xi(M)\leq M\leq\bar mC(T\xi)$ by \eqref{Phiprimetrivial},
so \eqref{Irecursion} holds. For $M>\bar mC(T\xi)$,
split the integral defining $I_\xi(M)$ at
$\bar mC(T\xi)$. By \eqref{Phiprimetrivial},
$$
\int_0^{\bar mC(T\xi)}|\Phi'_\xi(t)|\dif t\le\bar mC(T\xi).
$$
For $t\ge\bar mC(T\xi)$, we have
$s=t/m_0\geq C(T\xi)$ (recall $m_0\le\bar m$), so
\eqref{Phipsileeta} applied at the environment $T\xi$ and
\eqref{fprimebound} yield
$$
f'_0\big(\big|\Phi_{T\xi}\big(t/m_0\big)\big|\big)
\le\bar p+\bar m\eta=:\theta_0.
$$
Hence, using \eqref{Phiderivative} and substituting $s=t/m_0$,
\begin{align*}
\int_{\bar mC(T\xi)}^M|\Phi'_\xi(t)|\dif t
&\leq\frac{\theta_0}{m_0}\int_{\bar mC(T\xi)}^M
\big|\Phi'_{T\xi}\big(t/m_0\big)\big|\dif t\\
&=\theta_0\int_{\bar mC(T\xi)/m_0}^{M/m_0}
|\Phi'_{T\xi}(s)|\dif s
\leq\theta_0I_{T\xi}(M),
\end{align*}
where in the last step we use $M/m_0\le M$. Adding the two parts proves \eqref{Irecursion}.

Iterate \eqref{Irecursion} $n$ times at fixed $M$: since $|\Phi'_\xi(t)|\leq1$, the remainder satisfies
$I_{T^n\xi}(M)\leq M$, hence
$$
I_\xi(M)
\leq\bar m\sum_{k=0}^{n-1}\theta_0^kC(T^{k+1}\xi)
+\theta_0^nM.
$$
Letting $n\to\infty$ at fixed $M$, and then $M\to\infty$, we get
\begin{equation}\label{Phiprimeintegrable}
\int_0^\infty|\Phi'_\xi(t)|\dif t
\le\bar m\sum_{k\ge0}\theta_0^kC(T^{k+1}\xi).
\end{equation}
By
$C(\xi)\le1+\eta^{-1/\beta}D_\beta(\xi)^{1/\beta}$, applying   \eqref{fourierconstant} of Lemma \ref{l2.3},  we have $\E C(\xi)^r<\infty$ for each
$r\geq1$. By stationarity $\E C(T^{k+1}\xi)^r=\E C(\xi)^r$, and
Minkowski's inequality yields
\begin{equation}\label{Aonemoment}
\Big(\E\Big[\sum_{k\ge0}\theta_0^kC(T^{k+1}\xi)\Big]^r\Big)^{1/r}
\le\sum_{k\ge0}\theta_0^k\big(\E C(\xi)^r\big)^{1/r}
=\frac{(\E C(\xi)^r)^{1/r}}{1-\theta_0}<\infty.
\end{equation}
Combining \eqref{Phiprimeintegrable} with \eqref{Aonemoment} we obtain
$\E A(\xi)^r<\infty$ for every $r\ge1$.

(ii) The finite measure $x\Pr_\xi(W\in\dif x)$ has total mass
$\Ex_\xi W=1$ and Fourier transform
$$
\int_0^\infty\e^{\im tx}\,x\Pr_\xi(W\in\dif x)
=-\im\Phi'_\xi(t),
$$
which is integrable by \eqref{Phiprimeintegrable}. Hence, Fourier inversion
 gives $x\Pr_\xi(W\in\dif x)$ a bounded continuous density. Dividing it by $x$ on
$(0,\infty)$, we conclude that the quenched law of $W$ admits a continuous
density $w_\xi$ on $(0,\infty)$ with
$$
\sup_{x>0}xw_\xi(x)
\leq\frac1{2\pi}\int_\R|\Phi'_\xi(t)|\dif t
=A(\xi).
$$
Let $q_\xi$ denote the bounded, uniformly continuous density of the weighted measure $x\Pr_\xi(W\in\dif x)$ obtained by Fourier inversion. It vanishes on $(-\infty,0]$ and satisfies $q_\xi(x)=xw_\xi(x)$ for $x>0$ and $\|q_\xi\|_\infty\leq A(\xi)$. Extend $w_\xi$ by zero to $(-\infty,0]$. By symmetry and Tonelli's theorem, for every $j\geq2$ and $x>0$,
$$
xw_\xi^{*j}(x)
=j\big(q_\xi*w_\xi^{*(j-1)}\big)(x)
\leq jA(\xi).
$$
Since $q_\xi$ is bounded and uniformly continuous and $w_\xi^{*(j-1)}$ is a probability density, the convolution on the right is continuous. Thus $w_\xi^{*j}$ is continuous on $(0,\infty)$.

(iii)
Let $\varphi_\xi(t)=\Ex_\xi\e^{-tW}$. We first prove that there exists
$s_0<1$, independent of $\xi$, such that
\begin{equation}\label{s0uniform}
\varphi_\xi(1)\le s_0.
\end{equation}
Fix $v\in(0,1)$. By \eqref{uniformpositive} and H\"older's
inequality,
$$
1=\Ex_\xi W
\le v+\Ex_\xi\big[W\1_{\{W>v\}}\big]
\le v+K^{1/p}\Pr_\xi(W>v)^{1-1/p},
$$
so, uniformly in $\xi$,
$$
\Pr_\xi(W>v)
\ge\Big(\frac{(1-v)^p}{K}\Big)^{1/(p-1)}=:d_v.
$$
Splitting at the event $\{W>v\}$,
\begin{align*}
\varphi_\xi(1)
=\Ex_\xi\e^{-W}
&\le\Pr_\xi(W\le v)+\e^{-v}\Pr_\xi(W>v)\\
&=1-(1-\e^{-v})\Pr_\xi(W>v)
\le1-(1-\e^{-v})d_v=:s_0<1,
\end{align*}
which proves \eqref{s0uniform}.

By
$\varphi_\xi(t)=f_0\big(\varphi_{T\xi}(t/m_0)\big)$ and
$f_i\le g$ (see \eqref{fdomg}), together with the monotonicity of $g$
and of $\varphi_{T^n\xi}$, we get, whenever $t/\Pi_n\ge1$,
\begin{equation}\label{Laplaceiteration}
\varphi_\xi(t)
\le g_n\big(\varphi_{T^n\xi}\big(t/\Pi_n\big)\big)
\le g_n\big(\varphi_{T^n\xi}(1)\big)
\le g_n(s_0).
\end{equation}
Applying Lemma \ref{l2.3n} with the fixed $z=1-s_0\in(0,1)$,
\begin{equation}\label{giterates0}
g_n(s_0)
=g_n\big(1-(1-s_0)\big)
\le C_0(1-s_0)^{-d_0}\rho^n
=:C\rho^n.
\end{equation}
For $t\ge1$, take $n=\lfloor\log t/\log\bar m\rfloor$. Then
$\Pi_n\le\bar m^n\le t$, so $t/\Pi_n\ge1$ and
\eqref{Laplaceiteration}--\eqref{giterates0} apply. Moreover, for $n>\log t/\log\bar m-1$,
$$
\rho^n
\le\rho^{-1}\exp\Big(-\frac{\log(1/\rho)}{\log\bar m}\log t\Big)
=\rho^{-1}t^{-\log(1/\rho)/\log\bar m}.
$$
Combining it with \eqref{Laplaceiteration}--\eqref{giterates0}, we obtain
\begin{equation}\label{Laplacedecay}
\varphi_\xi(t)
\le Ct^{-\log(1/\rho)/\log\bar m},
\qquad t\ge1,
\end{equation}
with  $C:=C_0(1-s_0)^{-d_0}\rho^{-1}$ independent of $\xi$.

Now choose $b>0$ small enough that
\begin{equation}\label{bchoice}
2b<\frac{\log(1/\rho)}{\log\bar m}
\qquad\text{and}\qquad
\theta:=\bar p\bar m^b<1.
\end{equation}
Hence
$\bar p\bar m^b\downarrow\bar p<1$ as $b\downarrow0$. For $x\le1$,
Markov's inequality applied to $\e^{-W/x}$ gives
$$
F_\xi(x)
=\Pr_\xi(W\le x)
=\Pr_\xi\big(\e^{-W/x}\ge\e^{-1}\big)
\le\e\,\varphi_\xi(1/x).
$$
Since $1/x\ge1$, by \eqref{Laplacedecay} and   \eqref{bchoice}, we conclude  \eqref{quenchedbounds}.

(iv)
  From $W\overset{d}{=}m_0^{-1}\sum_{j=1}^{Z_1}W^{(j)}$, we have
$$
F_\xi(x)=\sum_{j\geq1}p_j(\xi_0)F_{T\xi}^{*j}(m_0x).
$$
By (ii) and the convolution bound proved there, for every $0<u<v<\infty$ and $j\geq1$,
$$
\sup_{x\in[u,v]}m_0p_j(\xi_0)w_{T\xi}^{*j}(m_0x)
\leq\frac{jp_j(\xi_0)A(T\xi)}u.
$$
Since $\sum_{j\geq1}jp_j(\xi_0)=m_0<\infty$, the series of derivatives converges locally uniformly on $(0,\infty)$. The series of distribution functions also converges uniformly, since its $j$-th term is bounded by $p_j(\xi_0)$. We may therefore differentiate term by term and obtain
\begin{equation}\label{qdensityeq}
w_\xi(x)
=m_0\sum_{j\geq1}p_j(\xi_0)\,w_{T\xi}^{*j}(m_0x),
\qquad x>0.
\end{equation}
We shall twice apply the elementary convolution estimate of
Lemma \ref{l2.4}.  Applying \eqref{convolutionestimate} with $d=0$, $\alpha=2b$, and $A_0=A(\xi)$, we have
$$
w_\xi^{*j}(x)\leq CA(\xi)x^{2b-1},\qquad j\geq2,\quad x\leq1.
$$
The quenched density equation \eqref{qdensityeq} therefore gives, for $x\leq x_*:=1/\bar m$,
\begin{equation}\label{bound0}
w_\xi(x)\leq m_0p_1w_{T\xi}(m_0x)+CA(T\xi)x^{2b-1}.
\end{equation}
Multiply the preceding bound by $x^{1-b}$, and set
$$
u_k(y):=y^{1-b}w_{T^k\xi}(y),
\qquad
\theta_k:=p_1(\xi_k)m_k^b.
$$
Applying \eqref{bound0}  at the environment $T^k\xi$ with argument $y$ yields
\begin{equation}\label{urecursion}
u_k(y)\le\theta_k\,u_{k+1}(m_ky)+CA(T^{k+1}\xi)\,y^b,
\qquad y\le x_*.
\end{equation}
By \eqref{bchoice},
$\theta_k\le\bar p\bar m^b=\theta<1$.
 Fix $0<x<x_*$. Let $\tau:=\inf\{n\ge0:\Pi_nx\ge x_*\}$. Note that $\tau<\infty$ since
$\Pi_n\ge\underline m^n\to\infty$. As $\Pi_kx<x_*$ for every $k<\tau$,
we may iterate \eqref{urecursion} along $y=\Pi_kx$ for
$k=0,1,\ldots,\tau-1$, which gives
\begin{equation}\label{uexpansion}
u_0(x)
\le C\sum_{k=0}^{\tau-1}\Big(\prod_{i<k}\theta_i\Big)
A(T^{k+1}\xi)\,(\Pi_kx)^b
+\Big(\prod_{i<\tau}\theta_i\Big)u_\tau(\Pi_\tau x),
\end{equation}
where $\prod_{i<0}:=1$. On the right-hand side, $\prod_{i<k}\theta_i\le\theta^k$ and
$(\Pi_kx)^b\le x_*^b\le1$ for $k<\tau$. Moreover, by the definition of $\tau$, $\Pi_{\tau-1}x<x_*$, hence
$$
x_*\le\Pi_\tau x=m_{\tau-1}\Pi_{\tau-1}x<\bar m x_*=1,
$$
and, using $yw_{T^\tau\xi}(y)\le A(T^\tau\xi)$ from
step (ii), we have
\begin{equation}\label{uterminal}
u_\tau(\Pi_\tau x)
=(\Pi_\tau x)^{-b}\,(\Pi_\tau x)\,w_{T^\tau\xi}(\Pi_\tau x)
\le x_*^{-b}A(T^\tau\xi).
\end{equation}
Substituting \eqref{uterminal} into \eqref{uexpansion},
\begin{equation}\label{x1bbound}
x^{1-b}w_\xi(x)
\le C\sum_{k<\tau}\theta^kA(T^{k+1}\xi)
+Cx_*^{-b}\theta^\tau A(T^\tau\xi)
\le C'\sum_{k\ge0}\theta^kA(T^k\xi),
\end{equation}
where $C':= \max\{C,1\}(\theta^{-1}+x_*^{-b}) $ is deterministic. For $x\in[x_*,1]$, the
same final bound (with a possibly larger $C'$) follows directly from
$xw_\xi(x)\le A(\xi)$, since $x^{1-b}w_\xi(x)\le x_*^{-b}A(\xi)$ and
$A(\xi)$ is the $k=0$ term of the sum. Now put
\begin{equation}\label{Bxidef}
B(\xi):=  C' \sum_{k\ge0}\theta^kA(T^k\xi),
\end{equation}
so that $w_\xi(x)\le B(\xi)x^{b-1}$ for $0<x\le1$. By step (i),
$\E A(T^k\xi)^r=\E A(\xi)^r<\infty$ for every $r\ge1$, so Minkowski's
inequality yields
$$
\big\|B(\xi)\big\|_r
\le C' \sum_{k\ge0}\theta^k\big\|A(T^k\xi)\big\|_r
=\frac{  C' \big\|A(\xi)\big\|_r}{1-\theta}<\infty,
$$
which completes the proof of \eqref{quenchedbounds0} and
\eqref{regularitymoments0}. A second application of
\eqref{convolutionestimate}, now with $d=b$, $\alpha=2b$, $K_0=C$ and
$A_0=B(\xi)$, yields
  \eqref{conditionalconvolution}.
\end{proof}

\begin{lemma}\label{densityproperties}
The annealed distribution $\mu_W(\cdot)=\P(W\in \cdot)$ satisfies
\begin{enumerate}[label=(\arabic*)]
\item $\mu_W$ is absolutely continuous w.r.t. Lebesgue measure;
\item the density $w$ can be chosen to be continuous on $(0,\infty)$;
\item for every compact set $K \Subset (0,\infty)$, $\sup_{x\in K}w(x)<\infty$;
\item $\sup_{x>0}xw(x)\leq \dfrac{1}{2\pi}\int_\R |\Phi'(t)|\dif t<\infty$.
\end{enumerate}
\end{lemma}

\begin{proof}
The key analytic input is the integrability of $\Phi'$. Claims (2)--(4) then follow by Fourier inversion.

(1)  By \cref{quenchedregularity}, for almost every environment $\xi$ the quenched law $\mathcal{L}_\xi(W)$ is absolutely continuous on $(0,\infty)$. Hence, for every Lebesgue-null set $A$,
$$
\P(W\in A)=\E[\Pr_\xi(W\in A)]=0,
$$
so the annealed law $\mu_W$ is absolutely continuous, i.e., (1) holds.
\smallskip

In the following we prove (2)--(4).
Since $\E W=1<\infty$, $\Phi$ is continuously differentiable with $\Phi'(t)=\im\E[W\e^{\im tW}]$. The annealed identities are
\begin{equation}\label{Phi-t}
\begin{aligned}
\Phi(t)&=\E\left[f_0\big(\Phi_{T\xi}(t/m_0)\big)\right],\\
\Phi'(t)&=\E\left[\frac1{m_0}f'_0\big(\Phi_{T\xi}(t/m_0)\big)\Phi'_{T\xi}(t/m_0)\right].
\end{aligned}
\end{equation}
By \cref{quenchedregularity} and Tonelli's theorem,
\begin{equation}\label{PhiL1}
\int_\R|\Phi'(t)|\dif t
\leq\E\int_\R|\Phi'_\xi(t)|\dif t
=2\pi\E A(\xi)<\infty.
\end{equation}
Now define
$$
\Psi(x)=\dfrac{1}{2\pi}\int_\R \e^{-\im tx}\Phi'(t)\dif t .
$$
By \eqref{PhiL1}, $\Psi$ is bounded and continuous with $|\Psi(x)|\leq\dfrac{1}{2\pi}\int_\R|\Phi'(t)|\dif t$. Since $\Phi'$ is the Fourier transform of the finite complex measure $\im x\mu_W(\dif x)$, the Fourier inversion theorem for finite measures gives $\im x\mu_W(\dif x)=\Psi(x)\dif x$. As $\mu_W$ is supported on $(0,\infty)$, it follows that $\mu_W$ admits the density
\begin{equation}\label{wPsi}
w(x)=\dfrac{\Psi(x)}{\im x}=\dfrac{-\im \Psi(x)}{x},
\qquad x>0,
\end{equation}
which is continuous on $(0,\infty)$ and satisfies
$$
\sup_{x>0}xw(x)=\sup_{x>0}|\Psi(x)|\leq\dfrac{1}{2\pi}\int_\R|\Phi'(t)|\dif t<\infty .
$$
This proves (2) and (4).

Finally, if $K\Subset (0,\infty)$ and $\delta_K:=\min_{x\in K}x>0$, then by \eqref{wPsi},
$$
\sup_{x\in K}w(x)\leq \dfrac{1}{2\pi \delta_K}\int_\R|\Phi'(t)|\dif t<\infty,
$$
which proves (3).
\end{proof}

\section{The proof of Theorem \ref{mainthm}}\label{sec:main}

\begin{proof}
The first assertion follows from Lemma \ref{densityproperties}. In the following we prove \eqref{main00}.

 \textbf{Step 1. Asymptotics of $F$.}
By the decomposition \eqref{decomposition} and independence of the first environment from the future environment, for $x>0$,
\begin{equation}\label{Fdecom}
F(x)=\E[p_1F(m_0x)]+R_F(x),\qquad
R_F(x):=\E\left[\sum_{j\geq2}p_jF_j(m_0x)\right].
\end{equation}
Define
\begin{equation}\label{Gt}
\G(t):=\e^{a_0 t}F(\e^{-t}),\qquad t\in\R.
\end{equation}
Since
$$
\E[p_1 F(m_0\e^{-t})]=\e^{-a_0 t}\E[p_1 m_0^{a_0}\G(t-\log m_0)]=\e^{-a_0 t}\tilde{\E}[\G(t-\log m_0)],
$$
multiplying \eqref{Fdecom} by $\e^{a_0t}$ yields the renewal equation on the whole line:
\begin{equation}\label{G}
\G(t)=\tilde{\E}[\G(t-\log m_0)]+H(t),
\qquad
H(t):=\e^{a_0t}R_F(\e^{-t}).
\end{equation}

We now solve \eqref{G}. First we bound   $H$. By \cite[Corollary 2.3]{grama2023}, for every $0<b<a_0$,
\begin{equation}\label{Fxbound}
C(b):=\E W^{-b}<\infty,\qquad \mbox{and}\qquad  F(x)\leq C(b)x^b.
\end{equation}
For the nonlinear term we instead use two initial individuals in the same environment. Since
$$
\E[p_1^2m_0^{a_0}]\leq\bar p\E[p_1m_0^{a_0}]=\bar p<1
$$
and $m_0$ is bounded, there is $a_2'>a_0$ with $\E[p_1^2m_0^{a_2'}]<1$. By \cite[Theorem 2.1(b)]{grama2023}, $C_2:=\E_2W^{-a_2'}<\infty$. Then $F_j(y)\leq F_2(y)\leq C_2y^{a_2'}$ for all $j\geq2$. Hence
\begin{equation}\label{RF}
R_F(x)\leq C_2x^{a_2'}\E[(1-p_1)m_0^{a_2'}]\leq C'x^{a_2'},
\end{equation}
where $C'=C_2\bar{m}^{a_2'}$. As $R_F(x)\leq1$ for all $x>0$, with $c:=\min\{a_2'-a_0,a_0\}>0$ we obtain
\begin{equation}\label{Ht}
H(t)\leq C\e^{-c|t|},\qquad t\in\R.
\end{equation}

Recall  that $\tilde{U}(B):=\sum_{k\geq 0}\tilde{\P}(S_k\in B)$ is the renewal measure of $(S_k)$ under $\tilde{\P}$.
Iterating \eqref{G} $n$ times gives
$$
\G(t)=\tilde{\E}[\G(t-S_n)]+\sum_{k=0}^{n-1}\tilde{\E}[H(t-S_k)].
$$
By \eqref{Fxbound}, $\G(t)\leq C(b)\e^{(a_0-b)t}$, and $\E[p_1m_0^b]<\E[p_1m_0^{a_0}]=1$ for $b<a_0$ since $m_0>1$ a.s., so
$$
\tilde{\E}[\G(t-S_n)]\leq C(b)\,\e^{(a_0-b)t}\left(\E[p_1m_0^b]\right)^n\longrightarrow 0,
\qquad n\to\infty.
$$
The series below converges, since by \eqref{Ht} and \eqref{renewal0},
$$
\int_0^\infty|H(t-s)|\,\tilde U(\dif s)
\leq C\left(1+\dfrac{1}{\log\underline m}\right)\sum_{j\geq0}\;\sup_{s\in[j,j+1]}\e^{-c|t-s|}<\infty .
$$
Therefore
\begin{equation}\label{Grenewalsum}
\G(t)=\sum_{k=0}^\infty\tilde{\E}[H(t-S_k)]
=\int_0^\infty H(t-s)\,\tilde{U}(\dif s),
\qquad t\in\R.
\end{equation}
Moreover, $H$ is continuous   by \eqref{Fdecom},
the continuity of $F$ and dominated convergence, so \eqref{Ht}
shows that $H$ is directly Riemann integrable (DRI). Applying Lemma \ref{lemma-renewal} by  combining \eqref{Grenewalsum}, \eqref{renewal00} and \eqref{tail}, we arrive at
$$
\lim_{t\to\infty}\G(t)=\frac{1}{\tilde{\mu}}\int_\R H(s)\,\dif s=:D,
$$
which, in view of \eqref{Gt}, means
\begin{equation}\label{Fx0}
  x^{-a_0}F(x) \longrightarrow  D,
\qquad x\downarrow 0.
\end{equation}
Finally we show that the constant $D$ is positive and finite. By \eqref{RF} and $R_F\leq1$,
$$
\int_\R H(s)\dif s=\int_0^\infty x^{-a_0-1}R_F(x)\dif x
\leq C'\int_0^1x^{a_2'-a_0-1}\dif x+\int_1^\infty x^{-a_0-1}\dif x<\infty.
$$
On the other hand, $\int_0^\infty y^{-a_0-1}F_j(y)\dif y=a_0^{-1}\E_jW^{-a_0}$. Tonelli's theorem and the substitution $y=m_0x$ therefore give
$$
\int_\R H(s)\dif s
=\frac1{a_0}\sum_{j\geq2}\E[p_jm_0^{a_0}]\,\E_jW^{-a_0}.
$$
Consequently,
\begin{equation}\label{Dformula}
D=\frac1{a_0\tilde\mu}\sum_{j\geq2}\E[p_jm_0^{a_0}]\,\E_jW^{-a_0}.
\end{equation}
Every summand in \eqref{Dformula} is nonnegative, and at least one is strictly positive: $\E[p_jm_0^{a_0}]>0$ for some $j\geq2$ since $\bar p<1$, and the martingale limit under $\P_j$ is finite and strictly positive a.s. Hence $D>0$.  Thus \eqref{Fx0} gives $F(x)\sim Dx^{a_0}$ as $x\downarrow0$, completing the proof of Step 1.

\medskip
\textbf{Step 2. The renewal equation for the density $w$.}\quad
For $j\geq1$, define $w_j(x):=\E[w_\xi^{*j}(x)]$. In particular, $w_1=w$.  By symmetry and \cref{quenchedregularity},  for $j\geq2$, $$xw_\xi^{*j}(x)=j\int yw_\xi(y)w_\xi^{*(j-1)}(x-y)\dif y\leq jA(\xi).$$ Hence,  $xw_j(x)\leq j\E A(\xi)$.
The continuous versions of $w_\xi^{*j}$ constructed
in \cref{quenchedregularity} satisfy, on any $[u,v]\subset(0,\infty)$,
$$
\sup_{x\in[u,v]}w_\xi^{*j}(x)\leq\frac{jA(\xi)}u,
\qquad \E A(\xi)<\infty,
$$
and by dominated convergence, each $w_j$ is continuous on $(0,\infty)$, and
$$
\sup_{x\in[u,v]}m_0\sum_{j\geq2}p_jw_j(m_0x)
\leq\frac{\E A(\xi)}u\sum_{j\geq2}jp_j
\leq\frac{\bar m\E A(\xi)}u.
$$
Applying dominated convergence again and noting that $\sum_j j p_j=m_0\le \bar{m}<\infty$ we have that
 the function
$$
R(x):=\E\left[m_0\sum_{j\geq2}p_jw_j(m_0x)\right]
$$
is continuous on $(0,\infty)$, and
\begin{equation}\label{globalRbound}
R(x)\leq\frac{\E A(\xi)}x\E\left[\sum_{j\geq2}jp_j\right]\leq\frac{\bar m\E A(\xi)}x.
\end{equation}
  By Tonelli's theorem,
\begin{align*}
R_F(x)
&=\E\left[\sum_{j\geq2}p_j\int_0^{m_0x}w_j(y)\dif y\right]
=\int_0^xR(u)\dif u,\\
\E[p_1F(m_0x)]
&=\int_0^x\E[m_0p_1w(m_0u)]\dif u.
\end{align*}
The last integrand is continuous in $u$, since
$$
m_0p_1w(m_0u)\le p_1 \E  A(\xi)/u\le \E A(\xi)/u
$$
by \cref{quenchedregularity}(ii)  after averaging, so it is locally bounded and dominated convergence applies.
Differentiating \eqref{Fdecom} gives, for every $x>0$,
\begin{equation}\label{wdecom}
w(x)=\E[m_0p_1w(m_0x)]+R(x).
\end{equation}
Let $b$ and $B(\xi)$ be as in \cref{quenchedregularity}. Choose $r>1$ sufficiently large that
$$
a_*:=b+a_0(1-1/r)>a_0.
$$
The asymptotic \eqref{Fx0} and the bound $F\leq1$ imply $F(x)\leq Cx^{a_0}$ for $0<x\leq1$. By \eqref{conditionalconvolution}, H\"older's inequality, and $0\leq F_\xi\leq1$,
\begin{equation}\label{wjz}
\begin{aligned}
w_j(x)&\leq Cx^{b-1}\E[B(\xi)F_\xi(x)]\\
&\leq Cx^{b-1}\big(\E B(\xi)^r\big)^{1/r}
\big(\E F_\xi(x)^{r/(r-1)}\big)^{1-1/r}\\
&\leq Cx^{b-1}F(x)^{1-1/r}\leq Cx^{a_*-1},
\qquad j\geq2,\quad 0<x\leq1.
\end{aligned}
\end{equation}
The constant is uniform in $j$. Thus
\begin{equation}\label{Rxxx0}
R(x)\leq Cx^{a_*-1}\E[(1-p_1)m_0^{a_*}]\leq C'x^{a_*-1},
\qquad 0<x\leq1/\bar m.
\end{equation}
Define
$$
h(x):=x^{1-a_0}w(x),\qquad
\mathcal H(t):= h(\e^{-t})=\e^{(a_0-1)t} w(\e^{-t}),\qquad
R_H(t):=\e^{(a_0-1)t}R(\e^{-t}).
$$
Equation \eqref{wdecom} becomes
\begin{equation}\label{Hdecom}
\mathcal H(t)=\tilde\E[\mathcal H(t-\log m_0)]+R_H(t).
\end{equation}
 Since $R$ is continuous, so is $R_H$.  By \eqref{globalRbound} and \eqref{Rxxx0}, $R_H$ satisfies
\begin{equation}\label{RHbound}
R_H(t)=O(\e^{-(a_*-a_0)t})\quad(t\to+\infty),\qquad
R_H(t)=O(\e^{a_0t})\quad(t\to-\infty).
\end{equation}
In particular, $R_H$ is directly Riemann integrable on $\R$.

\medskip
\textbf{Step 3. Asymptotics of $w$.}\quad
By \cref{densityproperties}, $w(x)\leq C/x$, so $\mathcal H(t)\leq C\e^{a_0t}$. Consequently
$$
0\leq\tilde\E[\mathcal H(t-S_n)]
\leq C\e^{a_0t}\big(\tilde\E\e^{-a_0\log m_0}\big)^n
=C\e^{a_0t}(\E p_1)^n\longrightarrow0.
$$
  The series $\sum_{k\geq0}\tilde\E[R_H(t-S_k)]$ converges by \eqref{RHbound} and \eqref{renewal0} (as in Step 1); iterating \eqref{Hdecom} and letting $n\to\infty$ therefore gives
\begin{equation}\label{densityrenewalsum}
\mathcal H(t)=\int_{[0,\infty)}R_H(t-s)\,\tilde U(\dif s).
\end{equation}
The key renewal theorem applies exactly as in Step 1, using \eqref{RHbound} and \eqref{renewal0} for the part on the negative half-line. Hence
$$
\mathcal H(t)\longrightarrow\frac1{\tilde\mu}\int_\R R_H(s)\dif s=:G<\infty.
$$
Finally, Tonelli's theorem and $R_F(x)=\int_0^xR(y)\dif y$ give
$$
\int_\R H(t)\dif t
=\int_0^\infty x^{-a_0-1}\int_0^xR(y)\dif y\dif x
=\frac1{a_0}\int_0^\infty y^{-a_0}R(y)\dif y
=\frac1{a_0}\int_\R R_H(t)\dif t.
$$
Therefore $G=a_0D>0$, and
$$
w(x)\sim a_0D x^{a_0-1},\qquad x\downarrow0.
$$
The proof is complete.

\end{proof}

\begin{remark}\label{rem:LF} We present an example for Theorem \ref{mainthm}.
Consider the linear fractional offspring distribution:
$$
p_0=0,\quad p_k(\xi_i)=q_i(1-q_i)^{k-1}, \quad k\geq 1,~~i\geq 0,
$$
where $q$ is a random variable taking values in $[\underline q,\bar q]$ for some constants $0<\underline q\leq\bar q<1$ (so that $m_0=1/q$ satisfies \cref{boundedm}), and $q_i$ are i.i.d. non-degenerate copies of $q$, with $-\log q$ non-lattice.
Writing $f_{0,n}:=f_0\circ\cdots\circ f_{n-1}$, we have
$$
f_{0,n}(s)=\dfrac{\Pi_n^{-1}s}{1-(1-\Pi_n^{-1})s},
$$
thus, conditionally on $\xi$, $Z_n$ has the geometric distribution with success probability $\Pi_n^{-1}$. Hence $\Ex_\xi[\e^{-uW}]=\lim_n f_{0,n}(\e^{-u/\Pi_n})=(1+u)^{-1}$, so the quenched limit law is $\mathrm{Exp}(1)$ for almost every environment. Therefore $W\sim\mathrm{Exp}(1)$ and $w(x)=\e^{-x}$.

For this model, $a_0=1$. Since $p_1m_0\equiv1$, we have $\tilde\P=\P$ and $\tilde\mu=\mu=\E\log m_0$. The sum of $j$ quenched copies has the $\mathrm{Gamma}(j,1)$ law for every environment. Thus
$$
\E[p_jm_0]=\E[(1-q)^{j-1}],\qquad
\E_jW^{-1}=\frac{\Gamma(j-1)}{\Gamma(j)}=\frac1{j-1}.
$$
Thus
$$
D=\dfrac{1}{\mu}\sum_{j\geq 2}\dfrac{\E[(1-q)^{j-1}]}{j-1}=\dfrac{1}{\mu}\E\left[\sum_{k\geq 1}\dfrac{(1-q)^k}{k}\right]=\dfrac{\E[-\log(q)]}{\mu}=\dfrac{\E\log m_0}{\mu}=1,
$$
which gives $G=1$. This agrees with the known $w(x)=\e^{-x}\sim 1=Gx^{a_0-1}$, $x\to 0$.
\end{remark}

{\bf Acknowledgement.}  The second author is supported by the National Natural Science Foundation of China  (Grant No. 12671167).

\end{document}